\documentclass{amsart}
\usepackage[american]{babel}
\usepackage{amssymb,amsmath,mathrsfs,pgf,tikz}
\usepackage{nicefrac}

\usepackage{hyperref}
\hypersetup{
  colorlinks  = true, 
  urlcolor    = black, 
  linkcolor   = black, 
  citecolor   = black 
}

\newcommand{\sd}{\triangle}

\newcommand{\NN}{\mathbb{N}}
\newcommand{\ZZ}{\mathbb{Z}}

\renewcommand{\H}{\mathcal{H}}

\newcommand{\C}{\mathcal{C}}
\newcommand{\J}{\mathcal{J}}
\newcommand{\B}{\mathcal{B}}
\newcommand{\K}{\mathcal{K}}
\renewcommand{\P}{\mathcal{P}}
\newcommand{\E}{\mathcal{E}}
\newcommand{\T}{\mathcal{T}}

\newcommand{\GG}{\mathscr{G}}
\newcommand{\TT}{\mathscr{T}}
\newcommand{\EE}{\mathscr{E}}

\renewcommand{\tt}{\mathbf{t}}
\newcommand{\xx}{\mathbf{x}}

\newcommand{\lt}{\operatorname{lt}}
\newcommand{\mup}{\mu_{\mathrm{p}}}

\newcommand{\reg}{\operatorname{reg}}
\newcommand{\supp}{\operatorname{supp}}

\newcommand{\set}[1]{\left \{ #1 \right \}}

\newcommand{\ts}{\textstyle}
\renewcommand{\ss}{\scriptstyle}

\newtheorem{theorem}{Theorem}[section]

\newtheorem{lemma}[theorem]{Lemma}
\newtheorem{cor}[theorem]{Corollary}
\newtheorem{prop}[theorem]{Proposition}
\newtheorem*{claim*}{Claim}

\theoremstyle{definition}
\newtheorem{definition}[theorem]{Definition}

\newtheorem{example}[theorem]{Example}

\begin{document} 

\title[Degree and regularity of Eulerian ideals of hypergraphs]%
{Degree and regularity of\\Eulerian ideals of hypergraphs}

\author{J.~Neves}
\author{G.~Varej\~ao}
\address{Univ.~Coimbra, CMUC, Department of Mathematics, 3001-501 Coimbra, Portugal.}
\email{neves@mat.uc.pt \textit{and} g.varejao@mat.uc.pt}

\thanks{This work was partially supported by the 
Centre for Mathematics of the University of Coimbra - UIDB/00324/2020, 
funded by the Portuguese Government through FCT/MCTES. The work of the second author is funded 
by the Portuguese FCT grant 2021.05420.BD}

\keywords{Binomial ideal, degree, Castelnuovo--Mumford regularity, hypergraph}
\subjclass[2010]{13A02, 13P10, 13P25, 05E40; 05C65, 05C70}

\begin{abstract}
We define the Eulerian ideal of a $k$-uniform hypergraph and study its degree and Castelnuovo--Mumford regularity.
The main tool is a Gr\"obner basis of the ideal obtained combinatorially from the hypergraph. We define the notion 
of parity join in a hypergraph and show that the regularity of the Eulerian ideal is equal to the maximum cardinality 
of such a set of edges. The formula for the degree involves the cardinality of the set of sets of vertices, $T$,  
that admit a $T$-join. We compute the degree and regularity explicity in the cases of a complete $k$-partite hypergraph 
and a complete hypergraph of rank $3$.
\end{abstract}
\maketitle

\section{Introduction}
Eulerian ideals of graphs were introduced in \cite{joinsAndEars}, motivated  
by the notion of vanishing ideals of projective toric sets parameterized by graphs, the study of which started in \cite{ReSiVi11}. 
Both the Eulerian ideal of a graph and the vanishing ideal of the toric subset parameterized by a graph 
are homogenous binomial ideals of the polynomial ring on the edges of the graph and yield one-dimensional, Cohen--Macaulay 
quotients. These properties are favourable to the study of the Castelnuovo--Mumford regularity and the degree
of these ideals. In the case of the vanishing ideal of the projective 
toric set parameterized by a graph, there has been substantial progress 
in the computation of these invariants (\emph{cf}.~\cite{GonRen, GoReSa13, parallel,nested, even, mariarafael}) and 
results indicate that they involve the cardinality of the base field. 
On the other hand, by \cite[Proposition 2.8]{joinsAndEars} we known that in the case of the Eulerian ideal, 
the regularity and the degree do not depend on the base field and therefore they reflect the combinatorics in a clear way. 
The degree of the Eulerian ideal has a simple formula involving 
the number of connected components of $G$ and the bipartite property (\emph{cf}.~\cite[Proposition 2.11]{joinsAndEars}). 
As for the formula of the regularity, the key combinatorial notion, introduced in \cite{neves}, 
is that of \emph{parity join}. These are subsets of edges of the graph which, given any even subgraph of the graph with an \emph{even}
number of edges (for example, two edge-disjoint triangles), have no more than half the number of edges of that 
subgraph in common with it (\emph{cf}.~Definition~\ref{def: parity join}, below). Together with 
a re-interpretation of this invariant in terms of $T$-joins with cardinalities of fixed parity, these 
have enabled a combinatorial formula for the Castelnuovo--Mumford regularity of the Eulerian ideal (or, equivalently,
the quotient it yields) for a general graph (\emph{cf}.~\cite[Theorem~4.13]{neves}).

\medskip

The purpose of the present article is to generalize the notions and results of \cite{joinsAndEars,neves} 
to hypergraphs. To obtain a homogeneous ideal without changing the original definition,
we will restrict to $k$-uniform hypergraphs, i.e., hypergraphs whose edges have cardinality equal to $k$.
We describe a Gr\"obner basis for the Eulerian ideal of a $k$-uniform hypergraph and from this basis 
we derive combinatorial formulas for the degree and the Castelnuovo--Mumford regularity. Finally, we apply
these to the explicit computation of these invariants in the case of complete hypergraphs. 
\medskip

The contents are the following. Sections~\ref{sec: ideal} and \ref{sec: T-joins} are devoted to 
preliminary material; we give the definition of the Eulerian ideal of a hypergraph 
(Definition~\ref{def: Eulerian ideal}), a characterization of the homogeneous binomials in the ideal 
(Proposition~\ref{prop: characterization of the binomials of the ideal}) and we 
define $T$-joins and even subsets in a hypergraph (Definition~\ref{def: T-join and even subsets}). In 
Section~\ref{sec: Grobner basis} we describe a Gr\"obner basis of the ideal (Theorem~\ref{thm: Grobner basis}) which 
we use in Section~\ref{sec: hilbert function} to describe the Hilbert function of the quotient 
by the ideal in terms of reduced parity joins (Definition~\ref{def: parity join} and 
Theorem~\ref{thm: formula for the Hilbert function}). In Sections~\ref{sec: regularity} and \ref{sec: degree}
we give combinatorial formulas for the invariants regularity (Theorem~\ref{thm: regularity}) and 
degree (Theorem~\ref{thm: degree}). In Section~\ref{sec: complete hypergraphs}, we  
compute explicitly these invariants for some families of complete hypergraphs.

\section{The ideal}\label{sec: ideal}

Let $\H$ be a hypergraph. More precisely, let $\H=(V_\H,E_\H)$ where 
$V_\H$ is a set (of vertices) and $E_\H$ is a set 
of subsets of $V_\H$, the elements of which we call edges of $\H$.
All hypergraphs in this article are assumed to be finite, with nonempty edge set and 
$k$-uniform, with $k\geq 2$, i.e., 
their edges have common cardinality equal to $k$. 
Simple graphs are identified with $2$-uniform hypergraphs. 
Fix $K$ a field and let $K[V_\H] = K[x_v : v\in V_\H]$ and $K[E_\H] = K[t_e : e\in E_\H]$ be polynomial rings.
Define a homomorphism of graded rings, $\varphi \colon K[E_\H]\to K[V_\H]$, by 
$\ts \varphi(t_e) = \prod_{v\in e} x_v$,
for all $e\in E_\H$. 
Throughout we will use the multi-index notation for monomials, i.e., whenever $\alpha \in \NN^{E_\H}$ is a nonnegative integer 
valued function of the set of edges, $\tt^\alpha \in K[E_\H]$ will denote the monomial:
$$
\ts \tt^\alpha = \prod_{e\in E_\H} t_e^{\alpha(e)}.
$$
If $\alpha \in \ZZ^{E_\H}$ is an integer-valued function of the edge set of $\H$, 
we will use $\supp(\alpha)$ to denote its support, i.e.,  
$\supp(\alpha) = \set{e\in E_\H : \alpha(e)\not = 0}$.
Similarly for functions of the vertex set $\gamma$ and monomials $\xx^{\gamma}\in K[V_\H]$.

\begin{definition}\label{def: Eulerian ideal}
Let $\H$ be a $k$-uniform hypergraph. With the above notations, the Eulerian ideal of $\H$ is the ideal 
$I_\H = \varphi^{-1} (x^2_v-x^2_w : v,w\in V_\H)$.
\end{definition}

Since $\varphi$ is a graded homomorphism and $I_\H$ is the preimage of a homogeneous binomial ideal, 
by a standard elimination argument, the Eulerian ideal of $\H$ is also a homogeneous binomial ideal 
(\emph{Cf}., for instance, \cite[Proposition~2.2]{joinsAndEars}). 
In the graph case, the name Eulerian ideal comes from the characterization of a Gr\"obner basis of this ideal 
in terms of the Eulerian subgraphs of the graph (\emph{cf}.~\cite[Theorem~3.3]{neves}). 
In the hypergraph case, the role of Eulerian subgraphs will be taken by the even subsets 
of edges (\emph{cf}.~Definition~\ref{def: T-join and even subsets}), as we will see later. 
This is equivalent in the case of graphs and, in the case of 
hypergraphs, is a direct approach that avoids considering any of the several notions of connectivity.
The motivation for this is the next result, which is a characterization 
of the binomials in the Eulerian ideal, along the lines of \cite[Proposition~2.5]{joinsAndEars}. 
Let us first fix some notation.

\begin{definition}\label{def: degree and supp2}
(i) If $C\subset E_\H$ and $v\in V_\H$, let $\deg_C(v)$ denote $\sum_{e\in C} |\set{v}\cap e|$.
(ii) If $\alpha\in \ZZ^{E_\H}$, let
$\supp_2 (\alpha)=\set{e\in E_\H : \alpha(e) \equiv_2 1}\subset E_\H$.
\end{definition}

\begin{prop}\label{prop: characterization of the binomials of the ideal}
Let $\tt^\alpha - \tt^\beta \in K[E_\H]$ be homogeneous and 
let $C = \supp_2(\alpha-\beta)$. Then $\tt^\alpha - \tt^\beta \in I_\H$
if and only if, for every $v\in V_\H$, $\deg_C(v)$ is even.
\end{prop}

\begin{proof}
Let $\varphi (\tt^\alpha - \tt^\beta) = \xx^\gamma - \xx^\mu$,
for some $\gamma,\mu \in \NN^{V_\H}$.
Then, for every $v\in V_\H$, 
$$
\renewcommand{\arraystretch}{1.3}
\begin{array}{l}
\deg_C(v) = \sum_{e\in C} |\set{v}\cap e|\\ 
\phantom{\deg_C(v) } \equiv_2 \sum_{e\in E_\H} |\set{v}\cap e|\cdot (\alpha(e)-\beta(e)) \\
\phantom{\deg_C(v) }
\equiv_2 \gamma(v)-\mu(v).
\end{array}
$$
Suppose that $\tt^\alpha-\tt^\beta \in I_\H$. Then $\xx^\gamma - \xx^\mu \in (x_v^2-x_w^2 : v,w \in V_\H)$ and therefore 
there exist 
$g_{vw}\in K[E_\H]$ such that 
$$
\ts \xx^\gamma - \xx^\mu = \sum\limits_{v,w\in E_\H} g_{vw} (x_v^2-x_w^2).
$$
Fix $v\in V_\H$. Setting all variables but $x_v$, in the expression above, equal to $1$ we get:
$$
x_v^{\gamma(v)}-x_v^{\mu(v)} = h(x_v^2-1) 
$$
for some $h\in K[x_{v}]$. Writing $h$ as linear combination of monomials in $x_v$ and 
working out the product on the left explicitly, we deduce that there exists $m\in \ZZ$ such 
that $\gamma(v) = \mu(v) + 2m$, as required.  
\medskip 

\noindent
Conversely, suppose that $\gamma(v)-\mu(v)$ is even, for every $v\in V_\H$.
Then there exists $\rho \in \set{0,1}^{V_\H}$ such that $\gamma = 2\gamma' + \rho$ and $\mu = 2\mu' + \rho$ 
for some $\gamma',\mu' \in \NN^{V_\H}$. Let us fix $u\in V_\H$
and let us work with the relations $x_v^2 \equiv x_u^2$ modulo the ideal $(x_v^2-x_w^2 : v,w \in V_\H)$. Then
$$
\ts \xx^{\gamma} - \xx^{\mu} \equiv x_u^{2|\gamma'|}\xx^\rho - x_u^{2|\mu'|}\xx^\rho,  
$$
where $|\gamma'| = \sum_{v\in V_\H} \gamma'(v)$ and $|\mu'| = \sum_{v\in V_\H} \mu'(v)$.
As $\xx^\gamma - \xx^\mu$ is homogeneous (because $\tt^\alpha-\tt^\beta$ is), we get 
$|\gamma'|=|\mu'|$ and thus the last binomial above is equal to zero.
We conclude that $\varphi(\tt^\alpha - \tt^\beta) \in (x_v^2-x_w^2 : v,w \in V_\H)$, i.e., $\tt^\alpha-\tt^\beta \in I_\H$.
\end{proof}

\begin{example}\label{exa: the bipyramid}
Let $\H$ be the $3$-uniform hypergraph on $V_\H=\set{1,2,3,4,5}$ with 
\begin{equation}\label{eq: the bipyramid edge set}
E_\H = \set{\set{1,2,4},\set{2,3,4},\set{1,3,4},\set{1,2,5},\set{2,3,5},\set{1,3,5}},
\end{equation}
depicted in Figure~\ref{fig: bipyramid} as the faces of a triangular bipyramid.
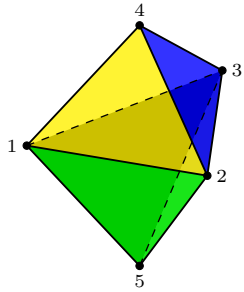
\begin{figure}[h]
\begin{tikzpicture}[scale=2]
\def\A{(0,0)}; \draw\A node[anchor = east]{$\ss 1$};
\def\B{(1.2,-.2)}; \draw\B node[anchor = west]{$\ss 2$};
\def\C{(1.3,.5)}; \draw\C node[anchor = west]{$\ss 3$};

\def\D{(.75,.8)}; \draw\D node[anchor = south]{$\ss 4$};
\def\E{(.75,-.8)}; \draw\E node[anchor = north]{$\ss 5$};

\fill[gray] \B -- \C -- \E -- \B;
\fill[black] \A -- \C -- \E -- \A;
\fill[yellow, opacity=.8] \A -- \B -- \D -- \A;
\fill[blue, opacity=.8] \B -- \C -- \D -- \B;
\fill[green, opacity=.8] \A -- \B -- \E -- \A;

\draw[line width = .7pt] \D -- \A;
\draw[line width = .7pt] \D -- \B;
\draw[line width = .7pt] \D -- \C;
\draw[line width = .7pt] \B -- \A;
\draw[line width = .7pt] \C -- \B;
\draw[dashed, line width = .5pt] \A -- \C;
\draw[line width = .7pt] \E -- \A;
\draw[line width = .7pt] \E -- \B;
\draw[dashed, line width = .5pt] \E -- \C;

\draw[fill=black]\A circle(.7pt);
\draw[fill=black]\B circle(.7pt);
\draw[fill=black]\C circle(.7pt);
\draw[fill=black]\D circle(.7pt);
\draw[fill=black]\E circle(.7pt);

\end{tikzpicture}
\label{fig: bipyramid}
\caption{A $3$-uniform graph.}
\end{figure}
Consider the set of faces 
\begin{equation}\label{eq: an even subset}
J=\set{\set{1,2,4},\set{2,3,4},\set{1,2,5},\set{2,3,5}}.
\end{equation}
One checks easily that $\deg_J(v)$ is even, for every $v\in V_\H$. Hence, , 
by Proposition~\ref{prop: characterization of the binomials of the ideal}, if 
$J=A\sqcup B$ is any partition of $J$ into parts of cardinality $2$, the binomial
$$
\ts \prod_{e\in A} t_e - \prod_{e\in B} t_e
$$
is an element of $I_\H$. For example, $t_{124}t_{234} - t_{125}t_{235}\in I_\H$.
There are $6$ binomials coming from these choices. In a minimal generating set, which
one can compute using Macaulay2, there are $8$ binomials of this type (of the $18$ 
we get by varing $J$ around the bipyramid, only $8$ are linearly independent) and there are 
$5$ other binomials of the form $t_e^2-t_f^2$, for $e,f\in E_\H$.
\end{example}

\section{$T$-joins}\label{sec: T-joins}

The characterization of the binomials in $I_\H$ given in Proposition~\ref{prop: characterization of the binomials of the ideal}
leads us to the notion of $T$-join in hypergraphs. This notion will play a key role in the computation of the regularity and the
degree of $K[E_\H]/I_\H$.

\begin{definition}\label{def: T-join and even subsets}
(i) If $T\subset V_\H$ and $J\subset E_\H$, we say that $J$ is a $T$-join if and only if 
$$
\ts \deg_J(v) = \sum_{e\in J} |\set{v}\cap e|\;\;  \text{is odd} \iff v\in T.
$$
In particular, a $\emptyset$-join is a subset of edges $J\subset E_\H$ such that $\deg_J(v)$
is even, for every $v\in V_\H$. These subsets of edges are called \emph{even}.
(ii) Let us denote the set of all even subsets of $E_\H$ by $\E(E_\H)$. 
(iii) Let us denote by $\T(V_\H)$ the set of all $T\subset V_\H$ 
for which there exists at least a $T$-join.
\end{definition}

Given $T\subset V_\H$, the question of existence of a $T$-join is pertinent. 
When $T=\emptyset$ or when $T$ is the set of vertices of a single edge, then 
a $T$-join always exists, namely the empty set and the singleton of the edge in question, 
respectively. In general, not every 
subset $T\subset V_\H$ admits a $T$-join. 
It is well-known that for $k=2$, i.e., for graphs, if a $T$-join exists for a given $T\subset V_\H$ 
then $|J\cap E_\C|$ must be even, for every connected component $\C$
of the graph $\H$. The converse also holds (\emph{cf.~\cite[Proposition~12.7]{KoVy06}}).
The situation for odd $k$ is different. For example in the hypergraph of Example~\ref{exa: the bipyramid}
the set of edges $J=\set{\set{1,2,4},\set{1,3,4},\set{2,3,4}}$, corresponding to the pyramid on the top, 
is a $\set{4}$-join.
\medskip 

The following is a minimal set of restrictions on the elements of $\T(V_\H)$, which we will use 
in the remainder of this article. 

\begin{prop}\label{prop: restrictions on T sets}
Let $T\in \T(V_\H)$ and let $J$ be a $T$-join. 
(i) If $k$ is even then $|T|$ is even.
(ii) If $k$ is odd then $|T|\equiv_2 |J|$.
In particular, when $k$ is odd there are no elements of $\E(E_\H)$ with odd cardinality. 
\end{prop}

\begin{proof}
Since
$$
\ts \sum_{v\in V_\H}\deg_J(v)= \sum_{v\in T}\deg_J(v) + \sum_{v\not \in T}\deg_J(v)= k|J|,
$$
if $k$ is even, then $|T|\equiv_2 0$ and if $k$ is odd, then  
$|T|\equiv_2 |J|$.
\end{proof}

Many properties of $T$-joins in graphs extend to hypergraphs. One example 
is the relation of $T$-joins with the symmetric difference of sets, which we denote 
by $\sd$. Namely, the fact that if $J_1$ is a $T_1$-join and $J_2$ is a $T_2$-join then $J_1\sd J_2$ 
is a $(T_1\sd T_2)$-join. This follows from the observation that, for every $v\in V_\H$, 
$$
\ts \deg_{J_1\sd J_2}(v) = \sum_{e \in J_1} |\set{v}\cap e| + \sum_{e \in J_2} |\set{v}\cap e| -2\sum_{e \in J_1\cap J_2} |\set{v}\cap e|.
$$
In particular, $\E(E_H)$, the set of even subsets of $E_\H$ is closed 
under the symmetric difference.

\section{A Gr\"obner basis}\label{sec: Grobner basis}

\begin{definition}\label{def: Eulerian binomials}
(i) Denote $\TT = \{t_e^2 - t_f^2 : e,f\in E_\H\}$. (ii)
A binomial $\tt^\alpha - \tt^\beta$, is said an Eulerian binomial if 
$\tt^\alpha$ and $\tt^\beta$ are square-free, coprime, of same degree and 
$\supp(\alpha)\sqcup \supp(\beta)\subset E_\H$ belongs to $\E(E_\H)$. 
Let $\EE$ denote the finite set of all Eulerian binomials. 
(iii) Let us denote $\GG = \TT \cup \EE$.
\end{definition}

Note that, here, the even subset $\supp(\alpha)\sqcup \supp(\beta)\subset E_\H$ has even cardinality, since 
$\tt^\alpha$ and $\tt^\beta$ are square-free, coprime and of the same degree.

\begin{prop}\label{prop: set of binomials belongs to the ideal}
$\GG = \TT \cup \EE\subset I_\H$.     
\end{prop}

\begin{proof}
Let $\tt^\alpha-\tt^\beta \in \GG$ and denote $C=\supp_2(\alpha-\beta)$. 
By Proposition~\ref{prop: characterization of the binomials of the ideal}, we must show that 
$C$ is even. If $\tt^\alpha-\tt^\beta \in \TT$ then $C=\emptyset$
and therefore $\deg_{\emptyset} (v) = 0$, for every $v\in V_\H$ (\emph{cf}.~Definition~\ref{def: degree and supp2}). 
If $\tt^\alpha-\tt^\beta \in \EE$
then 
$$
\supp_2(\alpha-\beta)=\supp(\alpha)\sqcup \supp(\beta)
$$
which, by definition, is even.
\end{proof}

From now on, fix a total order on the set $E_\H$ and consider the associated 
graded reverse lexicographic order on $K[E_\H]$. The next result gives a sufficient 
condition for a binomial to reduce to zero modulo the set of binomials 
$\GG = \TT\cup \EE$.

\begin{prop}\label{prop: suficient condition to reduce to zero}
Let $\tt^\alpha - \tt^\beta \in K[E_\H]$ be homogeneous. If 
$\supp_2(\alpha-\beta)$
is even, then, with respect to the 
graded reverse lexicographic order, 
$\tt^\alpha-\tt^\beta$ reduces to zero modulo $\GG = \TT \cup \EE$.
\end{prop}

\begin{proof} 
We will use complete induction on the degree of $\tt^\alpha-\tt^\beta$. If $\deg(\tt^\alpha-\tt^\beta)=0$, then
$\tt^\alpha-\tt^\beta = 0$ and there is nothing to show.
Assume that $\deg(\tt^\alpha-\tt^\beta)>0$. 
Let $\tt^\delta = \gcd(\tt^\alpha,\tt^\beta)$ and assume that $\tt^\delta \not = 1$. 
Let us write $\tt^\alpha = \tt^\delta \tt^\gamma$ and $\tt^\beta = \tt^\delta\tt^\mu$.
Then 
$$
\supp_2(\alpha-\beta) = \supp_2(\gamma-\mu).
$$
As $\deg(\tt^\gamma-\tt^\mu)< \deg(\tt^\alpha-\tt^\beta)$, by induction, 
$\tt^\gamma-\tt^\mu \stackrel{\GG}{\longrightarrow} 0$.
We may thus restrict to the case when $\tt^\alpha$ and $\tt^\beta$ are coprime.  
\medskip

\noindent
Assume, without loss of generality, that the leading term of $\tt^\alpha-\tt^\beta$ is $\tt^\alpha$.
If $\tt^\alpha$ and $\tt^\beta$ are both square-free, then $\tt^\alpha-\tt^\beta \in \EE$ and we have finished.
Suppose that $\tt^\alpha$ is square-free but there exists $\tt^\epsilon \not = 1$ 
such that $\tt^\beta = (\tt^\epsilon)^2 \tt^\eta$, for suitable square-free $\tt^\eta$.
Let $\tt^\sigma$ denote the product of the first $d$ edges in $\supp(\alpha)$, where $d=\deg(\tt^\epsilon)$, and consider 
$\tt^\alpha \tt^{-\sigma} - \tt^\eta \tt^\sigma$. Since $\deg(\tt^\alpha)=\deg(\tt^\beta)$ we deduce that 
$\tt^\alpha \tt^{-\sigma} - \tt^\eta \tt^\sigma$ is homogeneous. 
Furthermore, since $\tt^\alpha \tt^{-\sigma}$ and $\tt^\eta \tt^\sigma$ are square-free, coprime monomials and since 
$$
\supp(\alpha-\sigma)\sqcup \supp(\eta+\sigma) = \supp(\alpha)\sqcup \supp(\eta) = \supp_2(\alpha-2\epsilon-\eta)
$$
which, by assumption, is even, we conclude that $\tt^\alpha \tt^{-\sigma} - \tt^\eta \tt^\sigma$ is, in fact, an Eulerian binomial. 
As $\tt^\sigma$ is the product of a set of least edges in $\supp(\alpha)$, its leading term is $\tt^\alpha\tt^{-\sigma}$,
which divides $\tt^\alpha$, the leading term of $\tt^\alpha-(\tt^\epsilon)^2\tt^\sigma$. Using this Eulerian binomial in a one-step 
reduction, we get: 
$$
\tt^\alpha - (\tt^\epsilon)^2 \tt^\eta \stackrel{\EE}{\longrightarrow} \tt^\eta\bigl[(\tt^\sigma)^2- (\tt^\epsilon)^2\bigr].
$$
Hence we may reduce to the case that $\tt^\alpha$ is not square-free (since the leading term of $(\tt^\sigma)^2- (\tt^\epsilon)^2$
is certainly not square-free). Suppose then that there exists $e\in E_\H$ such that $t_e^2\mid \tt^\alpha$. Let $f\in E_G$ be an edge, least to 
$e$, such that $t_f\mid \tt^\beta$. Then $t_e^2-t_f^2$ has leading term $t_e^2$. Using this element
of $\TT$ in a one-step reduction, we get:
$$
\tt^\alpha-\tt^\beta \stackrel{\TT}{\longrightarrow} t_f(t_f\tt^\alpha t_e^{-2} - \tt^\beta t_f^{-1}).
$$
If $\gamma,\mu\in\NN^{E_G}$ are such that $\tt^\gamma = t_f\tt^\alpha t_e^{-2}$ and $\tt^\mu = \tt^\beta t_f^{-1}$
then 
$$
\supp_2(\gamma-\mu)=\supp_2(\alpha-\beta).
$$
Hence, by induction $\tt^\gamma-\tt^\mu \stackrel{\GG}{\longrightarrow} 0$.
\end{proof}

Next, let us use Propositions~\ref{prop: characterization of the binomials of the ideal} and 
\ref{prop: suficient condition to reduce to zero} to show that 
$\GG=\TT\cup \EE$ is a Gr\"obner basis for the Eulerian ideal with
respect to the graded reverse lexicographic order. 
This result generalizes \cite[Theorem~3.3]{neves}.

\begin{theorem}\label{thm: Grobner basis}
With respect to the graded reverse lexicographic order, $\GG = \TT \cup \EE$ is a Gr\"obner basis 
of the Eulerian ideal $I_\H$.     
\end{theorem}

\begin{proof}
We will use Buchberger's criterion and we start by showing that $I_\H = (\GG)$.
By Proposition~\ref{prop: set of binomials belongs to the ideal}, $(\GG)\subset I_\H$.
To prove the reverse inclusion, we will use the fact that $I_\H$ is generated by homogeneous binomials.
Let $\tt^\alpha-\tt^\beta$ be a homogeneous binomial in $I_\H$ and let
$C=\supp_2(\alpha-\beta)$. Then, by Proposition~\ref{prop: characterization of the binomials of the ideal},
$C$ is even and hence, by Proposition~\ref{prop: suficient condition to reduce to zero}, 
$\tt^\alpha-\tt^\beta$ reduces to zero modulo $\GG=\TT\cup \EE$, which implies that 
$\tt^\alpha-\tt^\beta \in (\GG)$.
\medskip 

\noindent
If $\tt^\alpha-\tt^\beta \in \GG$ 
then $\supp_2(\alpha-\beta)$ is even. To prove that $S(f,g)$ reduces to 
zero modulo $\GG$ this property of  $f,g \in \GG$ will suffice. 
Assume, without loss of generality, that 
$f=\tt^{\alpha}-\tt^\beta$ and $g= \tt^\gamma - \tt^\mu$
with $\lt(f) = \tt^\alpha$, $\lt(g)= \tt^\gamma$. Let us denote
$$
C_1 = \supp_2 (\alpha-\beta),\quad 
C_2 = \supp_2 (\gamma-\mu)
$$
and let $\tt^\delta = \gcd(\tt^\alpha,\tt^\gamma)$. 
The $S$-polynomial of $f$ and $g$ is equal to 
$\tt^{\alpha+\mu-\delta} - \tt^{\beta+\gamma-\delta}$, which is, of course, a homogeneous binomial. 
Let 
$$
\renewcommand{\arraystretch}{1.3}
\begin{array}{l}
C = \supp_2(\alpha+\mu-\delta-\beta-\gamma+\delta)\\
\phantom{J }= \supp_2(\alpha-\beta)\sd \supp_2(\gamma-\mu)\\
\phantom{J }= C_1\sd C_2.
\end{array}
$$
Since $C_1, C_2\subset E_\H$ are even subsets, we deduce that $C$ is even. 
Therefore, by Proposition~\ref{prop: suficient condition to reduce to zero}, 
$S(f,g)$ reduces to zero modulo $\GG$.
\end{proof}

\section{The Hilbert function}\label{sec: hilbert function}

As is well-known, a Gr\"obner basis of $I_\H$ enables an explicit characterization of a monomial basis
for the quotient $K[E_\H]/I_\H$. We will use this idea to determine a combinatorial 
formula for the Hilbert function and series of $K[E_\H]/I_\H$. The result we obtain here is a generalization 
of the corresponding result for the case of graphs \cite[Theorem~2.7]{codes}. Despite that  $K=\ZZ/3$ in \cite{codes}, 
because the emphasis there is on paramaetrized codes over graphs, this result holds over any field, 
as will be shown here. The key combinatorial invariant involved is the notion of \emph{parity 
join} for hypergraphs, introduced below, which was first defined in \cite{neves}, for graphs. 
This notion is related to the notion of \emph{join} (\emph{cf}.~\cite{frank}).

\begin{definition}
Let $I\subset K[x_1,\dots,x_n]$ be a homogeneous ideal in a polynomial ring endowed 
with a choice of a monomial order. Given $d\geq 0$, let $\B_d(I)$ denote the set of degree $d$ 
monomials that do not belong to the initial ideal of $I$.
\end{definition}

It is well known that the cosets with representative in $\B_d(I)$ form a $K$-basis for the degree $d$
component of $K[x_1,\dots,x_n]/I$. (\emph{Cf}., for example, \cite[Theorem~2.6]{EnHe12}). In particular, 
$\dim \bigl(K[x_1,\dots,x_n]/I\bigr)_d = |\B_d(I)|$.
From now on, let us denote by $\ell\in E_\H$ the least edge of $\H$. Consider $(I_\H,t_\ell^2)$. 
The quotient $K[E_\H]/(I_H,t_\ell^2)$ is an Artinian ring, because $t_e^2 \in (I_\H,t_\ell^2)$,
for every $e\in E_\H$. We will use this quotient to characterize the Hilbert function of $K[E_\H]/I_\H$.
Let us first make use of the Gr\"obner basis, $\GG$, of $I_\H$ given in Definition~\ref{def: Eulerian binomials},
to give an explicit characterization of the set $\B_d(I_\H,t_\ell^2)$.

\begin{lemma}\label{lemma: basis of Artinian quotient}
If $\GG$ is the Gr\"obner basis of $I_\H$ given in Definition~\ref{def: Eulerian binomials}, then 
$$\B_d(I_\H,t_\ell^2) = \bigl\{\tt^\gamma\in K[E_\H] : \deg(\tt^\gamma)=d\text{ and }
\lt(g)\nmid \tt^\gamma,\text{ for all }g\in \GG\cup \set{t_\ell^2}\bigr\}.$$
\end{lemma}

\begin{proof}
Given that $\GG$ is a Gr\"obner basis of $I_\H$ with respect to the graded 
reverse lexicographic order and no leading term of an element of $\GG$ is divisible by $t_\ell$ we deduce 
that $\GG\cup \set{t_\ell^2}$ is a Gr\"obner basis for $(I_\H,t_\ell^2)$, i.e., 
the ideal of leading terms of $(I_\H,t_\ell^2)$ is generated by $\GG\cup \set{t_\ell^2}$ and the result 
follows.
\end{proof}

\begin{example}
Let us go back to the $3$-uniform hypergraph of Example~\ref{exa: the bipyramid}.
Fix the order of $E_\H$ as given in \eqref{eq: the bipyramid edge set}, so that 
the last variable is $t_\ell=t_{135}$. The following is a list of the sets $\B_d(I_\H,t_\ell^2)$,
for $d\geq 0$, computed using Macaulay2.
$$
\begin{tabular}{l|l}
$d$ & $\B_d(I_\H,t_\ell^2)$\\
\hline 
$0$ & $\set{1}$ \\
$1$ & $\set{t_{124},\; t_{234},\; t_{134},\; t_{125},\; t_{235},\; t_{135}}$ \\ 
$2$ & $\set{t_{124}t_{235},\; t_{124}t_{135},\; t_{234}t_{135},\; t_{134}t_{135},\; t_{125}t_{235},\; t_{125}t_{135},\; t_{235}t_{135}}$ \\
$3$ & $\set{t_{124}t_{235}t_{135},\; t_{125}t_{235}t_{135}}$\\
$\geq 4$ & $\emptyset$
\end{tabular}
$$
Since, for every $e\not = \ell$, $t_e^2 -t_\ell^2 \in \GG$ has leading term equal to $t_e^2$, from 
Lemma~\ref{lemma: basis of Artinian quotient}, we deduce that the elements of $\B(I_\H,t_\ell^2)$ are square-free monomials, as 
we can check directly in this example. Hence all the monomials in $\B_d(I_\H,t_\ell^2)$
are in bijection with a certain set of subsets of $d$ edges of $\H$.
\end{example}

\begin{definition}\label{def: parity join}
(i) If $J\subset E_\H$, then $J$ is said a \emph{parity join} if 
$$
\ts |J\cap C|\leq \frac{|C|}{2},
$$
for all $C\in\E(E_\H)$, with $|C|$ \emph{even}. 
(ii) Let us denote 
$$\mup(\H) = \max \set{|J| : J\text{ is a parity join}}.$$
(iii) With respect to the total order of $E_\H$ we are fixing, $J$ is called
a \emph{reduced parity join} if $J$ is a parity join and, 
for every nonempty $C\in \E(E_\H)$ such that $|C|$ is even and $|J\cap C| = \frac{|C|}{2}$,
$J$ contains the least edge of $C$. We denote the set of reduced parity joins of 
cardinality $d$ by $\J_d^{\mathrm r}(\H)$.
\end{definition}

We will see next that $\B_d(I_\H,t^2_\ell)$ and $\J_d^{\mathrm r}(\H)$ are in bijection.
The maximum cardinality of a parity join, $\mup(\H)$, is related to an important invariant of 
$\K[E_\H]/I_\H$, the Castelnuovo--Mumford regularity; which is the topic of the next section.

\begin{theorem}\label{thm: characterization of standard monomials}
The map $\B_d(I_\H,t_\ell^2)\to \J_d^{\mathrm r}(\H)$ defined by $\tt^\gamma \mapsto \supp(\gamma)$ 
is well-defined and a bijection.
\end{theorem}

\begin{proof}
Let $C\subset E_\H$ be an even set of even cardinality. Suppose that $J=\supp(\gamma)$ 
satisfies $|J\cap C| > \frac{|C|}{2}$. Let $\tt^\alpha$ be the product of the first $\frac{|C|}{2}$
edges in $J$ and let $\tt^\beta$ denote the product of the remaining edges. Then 
$\tt^\alpha - \tt^\beta\in \EE$ has leading term equal to $\tt^\alpha$. 
This is also the case if $|J\cap C|=\frac{|C|}{2}$ and $J$ does not contain the least edge in $C$. 
Both cases lead to a contradiction. Since $\tt^\gamma$ is square free and thus 
$|J|=\deg(\tt^\gamma) = d$, we deduce that $J\in \J_d^{\mathrm r}(\H)$. Hence the map is well-defined. 
Using again the square-free property of the elements of $(I_\H, t^2_\ell)$ we see 
that the map is injective. 
\medskip 

\noindent
Let us now prove that the map is surjective. Let $J\subset E_\H$ be a reduced parity join and let $\tt^\gamma$
be the product of the edges in $J$. Since $\tt^\gamma$ is square-free, it suffices to show that 
$\tt^\gamma$ is not divisible by the leading term of an element in $\EE\subset \GG$.
Let us consider an Eulerian binomial $\tt^\alpha- \tt^\beta$ and assume, 
without loss of generality, that $\lt(\tt^\alpha- \tt^\beta)=\tt^\alpha$.
Then, this means that $\supp(\beta)$ contains the least edge of the corresponding
even set 
$$
C=\supp(\alpha)\sqcup \supp(\beta)\subset E_\H.
$$
Assume, with a view to a contradiction that $\tt^\alpha \mid \tt^\gamma$. Then 
$|J\cap C|\geq \deg(\tt^\alpha) = \frac{|C|}{2}$.
Since $J$ is a parity join, we deduce that $|J\cap C|= \frac{|C|}{2}$ 
and therefore $J\cap C = \supp(\alpha)$. But then, as $J$
is a reduced parity join, the least edge of $C$ must be in $J$, which is to say that 
it belongs to $\supp(\alpha)$. This is a contradiction. 
\end{proof}

Before we use Theorem~\ref{thm: characterization of standard monomials} to give a combinatorial 
formula for the Hilbert function of $K[E_\H]/I_\H$ we need to show that 
$t^2_\ell$ is a regular element for this quotient. 
We will show that, in fact, any monomial has this property. 
The proof can be taken almost \emph{verbatim} from the proof of \cite[Proposition~2.1]{joinsAndEars}.

\begin{lemma}\label{lemma: monomial is regular}
If $\tt^\gamma\in K[E_\H]$ is a monomial, then $\tt^\gamma$ is $\bigl(K[E_H]/I_\H\bigr)$-regular. 
\end{lemma}

\begin{proof}
Given the definition 
of $I_\H$ (\emph{cf}.~Definition~\ref{def: Eulerian ideal}), it suffices to 
show that if $x_u$ is any variable in $K[V_\H] = K[x_v : v\in V_\H]$ 
and $f\in K[V_\H]$, then 
$$
x_uf \in (x_v^2-x_w^2 : v,w\in V_\H) \iff f\in (x_v^2-x_w^2 : v,w\in V_\H).
$$
Fix a total order of the vertices of $\H$ for which $u\in V_\H$ is the least vertex. 
Then 
$$
\mathfrak{g} = \set{x_v^2-x_u^2 : v\in V_\H, v\not = u}
$$
is a Gr\"obner basis of $(x_v^2-x_w^2 : v,w\in V_\H)$ with the respect to the 
graded reverse lexicographic order.
Assume $x_uf \in (x_v^2-x_w^2 : v,w\in V_\H)$. Then $x_uf$ reduces to zero modulo 
$\mathfrak{g}$. As no leading term of an element of $\mathfrak{g}$ is divisible by $x_u$, 
we conclude that $x_u$ is a factor of every one-step reduction in the division algorithm, which is to say that $f$
reduces to zero modulo $\mathfrak{g}$ and hence $f\in (x_v^2-x_w^2 : v,w\in V_\H)$.
\end{proof}

\begin{theorem}\label{thm: formula for the Hilbert function}
$\dim_K \bigl( K[E_G]/I_\H \bigr )_d = \sum_{j\geq 0} |\J_{d-2j}^{\mathrm r}(\H)|$.
\end{theorem}

\begin{proof}
Since $t_\ell^2$ is $\bigl(K[E_\H]/I_\H\bigr)$-regular,  
the following is a short exact sequence of graded homomorphisms of $K[E_\H]$-modules:
\begin{equation}\label{eq: L742}
0\to \frac{K[E_\H]}{I_\H}[-2] \stackrel{\cdot t_\ell^2}{\longrightarrow} 
 \frac{K[E_\H]}{I_\H}  \to \frac{K[E_\H]}{(I_\H,t^2_\ell)}\to 0.
\end{equation}
Using Theorem~\ref{thm: characterization of standard monomials} and 
the above, we deduce that 
$$
\dim_K \bigl( K[E_G]/I_\H \bigr )_d = \dim_K \bigl( K[E_G]/I_\H \bigr )_{d-2} + |\J_{d}^{\mathrm r}(\H)|.
$$
Iterating, we obtain $\dim_K \bigl( K[E_G]/I_\H \bigr )_d = \sum_{j\geq 0} |\J_{d-2j}^{\mathrm r}(\H)|$.
\end{proof}

\begin{cor}
$K[E_\H]/I_\H$ is a $1$-dimensional and Cohen--Macaulay. 
\end{cor}

\begin{proof}
For $d>\mup(\H)$ the set $\J_d^{\mathrm r}(\H)$ is empty. Therefore, 
by Theorem~\ref{thm: formula for the Hilbert function}, 
the Hilbert polynomial is a nonzero constant and hence $\dim K[E_H]/I_H=1$. 
Since any monomial  
$\tt^\gamma$ is $\bigl(K[E_\H]/I_\H\bigr)$-regular, 
we deduce that $K[E_\H]/I_\H$ is Cohen--Macaulay.
\end{proof}

\begin{cor}\label{cor: Hilbert Series}
Let $s=|E_\H|$ and, for every $0\leq d\leq s$, let 
$j_d=|\J_d^{\mathrm r}(\H)|$ denote the number of reduced parity joins of cardinality $d$. The Hilbert series 
of $K[E_\H]/I_\H$, in the variable $z$, is equal to: 
\begin{equation}\label{eq: L700}
\frac{1+sz + j_2z^2 + \cdots +j_s z^s}{1-z^2}
\end{equation}
\end{cor}

\begin{proof}
Let us denote by $F(z)$ the Hilbert series of 
$K[E_\H]/I_\H$ and by $M$ the quotient $\frac{K[E_\H]}{(I_\H,t_\ell^2)}$. 
By \eqref{eq: L742} and Theorem~\ref{thm: characterization of standard monomials},
$$
\ts (1-z^2)F(z) = \sum_{d\geq 0} \bigl (\dim_K M_d\bigr ) z^d = \sum_{d\geq 0} j_d z^d.
$$
Since $j_0 = 1$, $j_1=|E_\H|=s$ and $j_d=0$, for all $d>s$, the formula follows.
\end{proof}

\section{Regularity}\label{sec: regularity}

As we saw in the previous section, the Hilbert function 
of $K[E_\H]/I_H$ is related to reduced parity joins. 
We will show that the Castelnuovo--Mumford regularity of this graded ring is related to the 
maximum cardinality of a reduced parity join. However there is a way to give this 
result in terms of the simpler notion of parity join and $\mup(\H)$ and that is what we will
do. The way to do this is to exploit the relation between parity joins and $T$-joins. 
The next result, which we use in the computation of the maximum cardinality  
is the analogue for hypergraphs of \cite[Lemmas~4.6 and 4.12]{neves}.

\begin{prop}\label{prop: parity joins as minimum cardinality fixed parity T-joins}
Let $T\in \T(V_\H)$ and let $J$ be a $T$-join. 
(i) There exist $T$-joins $J_1,J_2$ with $|J_1|\not \equiv_2 |J_2|$ if and 
only if there exists $C\in \E(E_\H)$  with $|C|$ odd. (ii) 
$J$ is a parity join if and only if $J$ has minimum cardinality among all $T$-joins 
$J'$ such that $|J'|\equiv_2 |J|$.
\end{prop}

\begin{proof}
(i) Let $J_1,J_2 \subset E_\H$ be $T$-joins with $|J_1|\not \equiv_2 |J_2|$. 
Then $C=J_1\sd J_2\subset E_G$ is an even set of odd cardinality. Conversely, 
if $C\in \E(E_\H)$ and $|C|$ is odd then $J\sd C$ is another $T$-join and 
$|J\sd C| \not \equiv_2 |J|$.
\medskip 

\noindent
(ii) Reasoning as before we deduce that   
\begin{equation}\label{eq: L988}
\set{J\sd C : C\in \E(E_\H) \text{ and $|C|$ is even}}.
\end{equation}
is the set of all $T$-joins, $J'\subset E_\H$, with $|J'|\equiv_2 |J|$.
If $C\in \E(E_\H)$ with $C$ even, then 
$$
\ts |J\cap C|\leq \frac{|C|}{2} \iff |J\sd C| \geq |J|.
$$
Therefore $J$ is a parity join if and only if $|J|$ is the minimum cardinality of an element of \eqref{eq: L988}.
\end{proof}

Let us denote the power set of $E_\H$ by $\P(E_\H)$. 
Since a set of edges may be identified with the corresponding monomial of $K[E_\H]$,
the graded reverse lexicographic order induces a total order on $\P(E_\H)$. 
Let us define it explicitely. Let $J_1\not = J_2\in \P(E_\H)$ 
and denote by $\tt^\alpha\not = \tt^\beta$ the square-free monomials
obtained as the products of the edges of $J_1$ and $J_2$, respectively. 
Then $\tt^\alpha \succ \tt^\beta$, in the graded reverse lexicographic order, if and only 
if either 
$$
\deg(\tt^\alpha) \succ \deg(\tt^\beta) \iff |J_1|>|J_2|
$$
or $\supp(\beta)$ contains the last edge of $\supp(\alpha-\beta)$, i.e., 
$J_2$ contains the last edge of the symmetric difference $J_1\sd J_2$. 

\begin{definition}\label{def: order on the set of parity joins}
If $J_1\not  = J_2 \in \P(E_\H)$ are elements of the power set of $E_\H$, we 
set $J_1\succ J_2$ if $|J_1|>|J_2|$ or, if $|J_1| = |J_2|$ and 
$J_2$ contains the last edge in $J_1\sd J_2$.
\end{definition}

It is clear that the partial order $\succeq$ defined by the above is a total order and, in particular, 
every (finite) subset of $\P(E_\H)$ has a minimum element 

\begin{prop}\label{prop: reduced parity joins as minima}
Let $T\in \T(V_\H)$ and let $J$ be a $T$-join. Then 
$J$ is a reduced parity join if and only if $J$ 
is the minimum with respect to $\succeq$ of the set 
\begin{equation}\label{eq: L899}
\set{J' : \text{$J'$ is a $T$-join and $|J'|\equiv_2 |J|$}}.
\end{equation}
\end{prop}

\begin{proof}
Let $J'$ be another $T$-join with $|J'|\equiv_2 |J|$ and denote $C=J'\sd J$ 
the corres\-ponding even cardinality element of $\E(E_\H)$. Then 
$$
\ts |J\cap C|= \frac{|C|}{2} \iff |J'| = |J|.
$$
Hence, if $J$ is a reduced parity join and $|J'|=|J|$ then $J$ contains the last 
edge of $C$ which implies that $J'\succ J$. Therefore $J$ is the minimum of \eqref{eq: L899}.
Conversely if $J$ is the minimum of this set then, by Proposition~\ref{prop: parity joins as minimum cardinality fixed parity T-joins},
$J$ is a parity join. Suppose that there exists $C\in \E(E_\H)$ nonempty and such that $|J\cap C|= \frac{|C|}{2}$. Set 
$J'=J\sd C$. Then $J'$ is another $T$-join and $|J'|=|J|$. Then, since $J'\succ J$, by definition, $J$ contains 
the last edge of $C$. We conclude that $J$ is a reduced parity join.
\end{proof}

Proposition~\ref{prop: parity joins as minimum cardinality fixed parity T-joins}
gives the existence of a reduced parity join of cardinality equal to that of any given 
parity join. In particular, it allows to compute $\mup(\H)$, 
as the maximum cardinality of a reduced parity join (\emph{cf}.~Definition~\ref{def: parity join}). 

\begin{cor}\label{cor: mup as maximum of cardinalities of reduced parity joins}
$\mup(\H)$ is the maximum cardinality of a reduced parity join. 
\end{cor}

The next theorem is a combinatorial formula for the regularity of $K[E_\H]/I_\H$.
This result is the generalization for hypergraphs of \cite[Theorem~4.13]{neves}.

\begin{theorem}\label{thm: regularity}
The regularity of $K[E_\H]/I_\H$ is equal to $\mup(\H)-1$.  
\end{theorem}

\begin{proof}
Since $K[E_\H]/I_\H$ is a $1$-dimensional, Cohen--Macaulay graded module, 
its regularity is equal to its index of regularity, i.e., the smallest degree 
$r$ such that $H(d)=P(d)$, for all $d\geq r$, where $H$ and $P$ denote the Hilbert function 
and the Hilbert polynomial of $K[E_\H]/I_\H$, respectively (\emph{cf}.~\cite[Corollary~4.8]{Ei05}).
In turn, the index of regularity is equal to the degree of the Hilbert series plus one (\emph{cf}.~\cite[Corollary~5.1.9]{Vi95}).
Since, by Corollary~\ref{cor: Hilbert Series}, the degree of 
the Hilbert series of $K[E_\H]/I_\H$ is equal to the maximum cardinality of a reduced 
parity join minus two and, in turn, by Propositions~\ref{prop: parity joins as minimum cardinality fixed parity T-joins} and 
\ref{prop: reduced parity joins as minima}, using Corollary~\ref{cor: mup as maximum of cardinalities of reduced parity joins}, 
the regularity of $K[E_\H]/I_\H$ is equal to $\mup(\H)-1$.
\end{proof}

\section{Degree}\label{sec: degree}

Since $\K[E_\H]/I_\H$ is a $1$-dimensional $K[E_\H]$-module its degree
coincides with its Hilbert polynomial, which, as the Hilbert series, \eqref{eq: L700}, indicates
is related to the cardinality of the set of reduced parity joins, $\sqcup_{d=0}^s \J_d^{\mathrm r}(\H)$.
In the case of graphs, because we know exactly which sets belong to $\T(V_\H)$, the degree 
of this module may be given in terms of the connected components of the graph and the bipartite property.
(\emph{Cf}.~\cite[Proposition~2.11]{joinsAndEars} and \cite[Proposition~4.10]{neves}.) In the general 
case, an alternative combinatorial characterization of the degree involves the cardinality 
of $\T(V_\H)$.

\begin{prop}\label{prop: bijections}
The map $\sqcup_{d=0}^s \J_d^{\mathrm r}(\H) \to \T(V_\H)$ sending a 
reduced parity join, $J\subset E_\H$, to the set 
$\set{v\in V_\H : \deg_J(v)\text{ is odd}}$ is a surjection.
If all elements of $\E(E_\H)$ have even cardinality then 
it is a bijection, otherwise, it is $2$-to-$1$.
\end{prop}

\begin{proof}
Fix $T\in \T(E_\H)$. By Proposition~\ref{prop: reduced parity joins as minima}, there exists 
a reduced parity join $J$ that maps to $T$. If all elements 
of $\E(E_\H)$ have even cardinality then, by Proposition~\ref{prop: parity joins as minimum cardinality fixed parity T-joins}
the set \eqref{eq: L899} is the full set of $T$-joins and, using again Proposition~\ref{prop: reduced parity joins as minima},
we deduce that there exists only one $T$-join which is a reduced parity join. If there exists 
$C\in \E(E_\H)$ of odd cardinality then $J$ and $J\sd C$ are $T$-joins with $|J\sd C| \not \equiv_2 |J|$.
Hence we may apply Proposition~\ref{prop: reduced parity joins as minima} twice to find two 
reduced parity joins mapping to $T$. As any other $T$-join must belong to one of the two 
corresponding sets of $T$-joins (the even cardinality ones and the odd cardinality ones) and, 
by Proposition~\ref{prop: reduced parity joins as minima}, only one in each set is a reduced parity join,
we deduce that the preimage of $T$ consists of exactly two reduced parity joins. 
\end{proof}

\begin{theorem}\label{thm: degree}
If $s=|E_\H|$, then the degree of $K[E_\H]/I_\H$ is equal to 
$\frac12\sum_{d=0}^s|\J^{\mathrm r}_ d(\H)|$.
Moreover, if no element of $\E(E_\H)$ has odd cardinality, then
the degree is $\frac12|\T(V_\H)|$, otherwise the degree is $|\T(V_\H)|$.
\end{theorem}

\begin{proof}
Since $K[E_\H]/I_\H$ is $1$-dimensional, the degree of this $K[E_\H]$-module
may be obtained by multiplying its Hilbert series by $(1-z)$ and setting $z=1$.
Using the rational form of the Hilbert series given in \eqref{eq: L700}, we deduce that 
$$
\ts \deg K[E_\H]/I_\H = \frac{1}{2}\sum_{d=1}^s j_d = \frac12 \sum_{d=0}^s|\J^{\mathrm r}(\H)|.
$$
The rest of the statement follows from Proposition~\ref{prop: bijections}.
\end{proof}

\section{Complete Hypergraphs}\label{sec: complete hypergraphs}

In this section we compute the degree and the regularity of 
$K[E_\H]/I_H$ when $\H$ is a complete $k$-partite graph or a $3$-uniform complete hypergraph.
The computations rely on the fact that in these cases the set $\T(V_\H)$ and 
the set of $T$-joins, for any given $T\in \T(V_\H)$, can be analysed explicitly, without 
using a notion of hypergraph connectivity.

\begin{definition}
A hypergraph is called a complete $k$-partite hypergraph if 
the vertex set, $V_\H$, is endowed with a $k$-partition, $V_\H=\sqcup_{i=1}^k V_i$, with 
$|V_i|=a_i>0$, such that $E_\H = \set{ \set{v_1,\dots,v_k} : v_i\in V_i, \text{ for every }i=1,\dots,k}$. 
Let us denote a complete $k$-partite hypergraph by $\K^k_{a_1,\dots,a_k}$.
\end{definition}

The next result is the analogue of the well-known result that a bipartite graph contains no cycles of 
odd cardinality. It is also an important characteristic of a complete $k$-partite hypergraph as regards
the results of the previous sections. 

\begin{prop}\label{prop: parity of even sets of complete k-partite hypergraphs}
If $\H = \K^k_{a_1,\dots,a_k}$ and $C\in\E(E_\H)$, then $|C|$ is even. 
\end{prop}

\begin{proof}
If $C\subset \E(E_\H)$, then, 
since every edge in $C$ contains a single vertex of $V_1$, 
$$
\ts |C| = \sum_{v\in V_1} \sum_{e\in C} |\set{v}\cap e| = \sum_{v\in V_1} \deg_C(v).
$$
Since $\deg_C(v)$ is even, for every $v\in V_\H$, this implies that $|C|$ is even. 
\end{proof}

\begin{prop}\label{prop: T sets of complete k-partite}
Let $\H = \K^k_{a_1,\dots,a_k}$ and $T\subset V_\H$. $T\in \T(V_\H)$ if and only if
$T_i=T\cap V_i$, for $i=1,\dots, k$, have cardinalities of equal parity. 
Moreover, if $T$ satisfies this condition, then the minimum cardinality of a $T$-join 
is $\max_{i=1}^k |T_i|$.
\end{prop}

\begin{proof}
 If $T\in \T(V_\H)$ and if $J$ is a $T$-join then,
for every $i=1,\dots,k$,
\begin{equation}\label{eq: L1034}
\ts |J| = \sum_{v\in T_i} \deg_J(v) + \sum_{v\in V_i\setminus T_i} \deg_J(v) \equiv_2 |T_i|.
\end{equation}
Hence $|T_i|$, for $i=1,\dots,k$, have equal parity. 
Conversely, let $T\subset V_\H$ be such that $T_i=T\cap V_i$, for $i=1,\dots,k$, have cardinalities 
of equal parity. Denote $r=\max_{i=1}^k |T_i|$. Consider the sequence of elements of  
$T_i$ written as $w_1^i,\dots,w_r^i$, where, the first $|T_i|$ in this sequence 
are the members of $T_i$ (without repetitions) and, if $|T_i|<r$,  
the last $r-|T_i|$ of them are equal to $w^i_{1}$. 
Consider the set of $r$ edges 
$$
J = \set{\set{w^1_i,\dots,w^k_i} : i=1,\dots,r}.
$$
If $v\not \in T$ then, clearly $\deg_J(v)=0$. If $v=w^i_j\in T_i$, with $2\leq j\leq |T_i|$ then 
$\deg_J(v)=1$. If $v=w^i_1\in T_i$ then, because 
$|T_i|\equiv_2 r$, $\deg_J(v) = r-|T_i|+1 \equiv_2 1$. 
We deduce that $J$ is a $T$-join of cardinality $r=\max_{i=1}^k |T_i|$.
It remains to be proved that this is the minimum cardinality of a $T$-join.
If $J'$ any $T$-join, then, applying the equality in \eqref{eq: L1034} to $J'$,
we get 
$$
\ts |J'|\geq  \sum_{v\in T_i} \deg_{J'}(v) \geq |T_i|,  
$$
for every $i=1,\dots,k$. We conclude that $|J'|\geq \max_{i=1}^k |T_i|$.
\end{proof}

\begin{theorem}
If $\H = \K^k_{a_1,\dots,a_k}$, then 
$$
\renewcommand{\arraystretch}{1.4}
\begin{cases}
\log_2 \bigl(\deg K[E_\H]/I_H\bigr) = \bigl(\sum_{i=1}^k a_i\bigr)-k \\ 
\reg K[E_\H]/I_\H = \max\set{a_1,\dots,a_k}-1.
\end{cases}
$$
\end{theorem}

\begin{proof}
As for the degree, using Theorem~\ref{thm: degree}, we only need to show that 
\begin{equation}\label{eq: L1054}
\ts \log_2|\T(V_\H)| = (\sum_{i=1}^k a_i)-k+1.
\end{equation}
But this is now straightforward from Proposition~\ref{prop: T sets of complete k-partite}.
Let us now deal with the statement on the regularity of $K[E_\H]/I_\H$.
By Theorem~\ref{thm: regularity} we must show that $\mup(\H) = \max\set{a_1,\dots,a_k}$. 
Let $J\subset E_\H$ be a parity join and let $T\in \T(V_\H)$ be the set of vertices, $v\in V_\H$,
such that $\deg_J(v)$ is odd. Since by Proposition~\ref{prop: parity of even sets of complete k-partite hypergraphs}
there are no elements of $\E(E_\H)$ of odd cardinality, using 
Proposition~\ref{prop: parity joins as minimum cardinality fixed parity T-joins},
we deduce that $|J|$ is the minimum cardinality of a $T$-join. Since we showed 
in Proposition~\ref{prop: T sets of complete k-partite} that the minimum cardinality 
of a $T$-join is $\max_{i=1}^k |T_i|$, where, recall, $T_i=T\cap V_i$, we conclude 
that $|J|\leq \max\set{a_1,\dots,a_k}$. This shows that 
$$
\mup(\H)\leq \max \set{a_1,\dots,a_k}.
$$
To prove the opposite inequality, it suffices to show that there exists a parity join of 
cardinality $\max \set{a_1,\dots,a_k}$. Arguing as before, this is the same as showing that 
there exists $T\in \T(V_\H)$ such that $\max\set{a_1,\dots,a_k} = \max_{i=1}^k |T_i|$. 
Let $1\leq m\leq k$ be such that $a_m = \max \set{a_1,\dots,a_k}$. If $a_m$ is even, fix  
$T = V_m$, if $a_m$ is odd, then, choosing $v_i\in V_i$, for all $i\not = m$, set 
$T = \set{v_1,\dots,v_{m-1}}\cup V_m \cup \set{v_{m+1},\dots,v_k}$. In both cases 
$T\in \T(V_\H)$ and $a_m = \max_{i=1}^k |T_i|$.
\end{proof}

\begin{definition}
A hypergraph is called complete of rank $k$ if $|V_\H|\geq k$ and 
$E_\H$ is the set of cardinality $k$ subsets of $V_\H$. Let us denote a complete 
hypergraph of rank $k$, with $n$ vertices, by $\K^k_{n}$.
\end{definition}

Recall that we are assuming throughout this article that $k\geq 2$.

\begin{prop}\label{prop: T sets of complete hypergraphs}
Let $\H=\K^k_n$. (i) If $n=k$ then $\T(V_\H) = \set{\emptyset,V_\H}$. 
(ii) If $n>k$ and $k$ is odd, then $\T(V_\H) = \P(V_\H)$.
(iii) If $n>k$ and $k$ is even, then $\T(V_\H)$ is equal to the set of elements of $\P(V_\H)$
of even cardinality. 
\end{prop}

\begin{proof}
(i) If $n=k$ and $V_\H= \set{v_1,\dots,v_n}$ then $E_\H = \set{\set{v_1,\dots,v_n}}$
and therefore a $T$-join exists if and only if $T=\emptyset$ or $T=\set{v_1,\dots,v_n}=V_\H$.
(ii) Assume that $n>k$ and $k$ is odd. To show that $\T(V_\H) = \P(V_\H)$ 
it suffices to show that $\set{v} \in \T(V_\H)$, for all $v\in V_\H$.
Fix $v_0\in V_\H$ and let $V$ be any subset of $V_\H\setminus \set{v_0}$ with cardinality  
$k$. Consider:
$$
J = \set{\set{v_0}\cup J' : J'\subset V\text{ and }|J'|=k-1}\subset E_\H.
$$
As $\deg_J(v_0)=\binom{k}{k-1}=k$ is odd, 
$\deg_J(w) = \binom{k-1}{k-2} = k-1$ is even, for every $w\in V$, and 
$\deg_J(v) = 0$ for every $v\not \in V\cup \set{v_0}$, we deduce that 
$$
\deg_J(v) \text{ is odd} \iff v\in \set{v_0}.
$$
In other words, $J$ is $\set{v_0}$-join. (iii) Assume that $n>k$ and $k$ is even. Let $T\in \T(V_\H)$ and let $J$ be a $T$-join. 
By Proposition~\ref{prop: restrictions on T sets}, $|T|$ is even. To prove that 
$\T(V_\H)$ is equal to the set of subsets of $V_\H$ of even cardinality, it suffices 
to show that, for every pair of distinct vertices $v_1,v_2\in V_\H$, 
$\set{v_1,v_2}\in \T(V_\H)$. Let us fix $v_1,v_2\in V_\H$. 
Choose $V\subset V_\H\setminus \set{v_1,v_2}$ of
cardinality $k-1$. Consider 
$$
J = \set{\set{u,v}\cup J' : J'\subset V\text{ and }|J'|=k-2}\subset E_\H.
$$
Arguing as before we deduce that $J$ is a $\set{v_1,v_2}$-join.
\end{proof}

\begin{cor}\label{cor: degree of complete hypergraph}
Let $\H=\K_n^k$. (i) If $n=k$, then $\deg K[E_\H]/I_H$ is equal to $1$. 
(ii) If $n>k$, then $\deg K[E_\H]/I_H$ is equal to $2^{n-1}$.
\end{cor}

\begin{proof}
(i) If $n=k$ then $|E_H|=1$ hence $I_\H = (0)$ and $K[E_\H]$ is $1$-dimensional. 
(ii) There are two cases. Assume first that $k$ is odd. Then, by Proposition~\ref{prop: restrictions on T sets},
no element of $\E(E_\H)$ has odd cardinality. Therefore, by Theorem~\ref{thm: degree},
the degree of $K[E_\H]/I_\H$ is equal to $\frac{|\T(V_\H)|}{2}$, which, by Proposition~\ref{prop: T sets of complete hypergraphs},
is equal to $2^{n-1}$. Let us now assume that $k$ is even. Let $V\subset V_\H$ be a subset of 
vertices of cardinality $k+1$ and 
let $J$ be the set of edges with vertices in $V$. Then 
$$
\ts \deg_J(v) = \binom{k}{k-1} = k 
$$
is even, for every $v\in V$ while $\deg_J(v)=0$, for every $v\in V_\H \setminus V$. 
In other words, $J$ is an even subset of edges. Since $|J|=\binom{k+1}{k} = k+1$ we conclude that 
$\E(E_\H)$ contains elements of odd cardinality. Hence the degree of $K[E_\H]/I_\H$ is 
$|\T(E_\H)|$, which, by Proposition~\ref{prop: T sets of complete hypergraphs}, is equal to 
$2^{n-1}$.
\end{proof}

\begin{theorem}
Let $\H=\K_n^3$. If $n=3$, then $\reg K[E_\H]/I_\H = 0$. 
If $n=4$, then $\reg K[E_\H]/I_\H = 4$. 
If $n\geq 5$, then $\reg K[E_\H]/I_\H = \lfloor\frac{n+1}{3}\rfloor $.
\end{theorem}

\begin{proof}
If $n=3$ then $|E_\H|=1$ and therefore, by Theorem~\ref{thm: Grobner basis}, $I_\H = (0)$, 
and hence $\reg K[E_\H]/I_\H = 0$. Assume now that $n=4$ and let $C\in \E(E_\H)$. 
By Proposition~\ref{prop: restrictions on T sets}, $|C|$ must be even and hence 
$|C|=0$, $2$ or $4$. A set of two distinct edges yields two vertices of degree
$2$ and two of degree $1$. A set of four edges is the whole of $E_\H$,
which yields all vertices of degree $3$. Hence we must have $|C|=0$. In other words,
$\E(E_\H) = \set{\emptyset}$ and, consequently, all subsets of $E_\H$ are parity joins. 
Therefore $\mup(\H) = |E_\H|=4$ and, by Theorem~\ref{thm: regularity},  
$\reg K[E_\H]/I_\H = 3$.
\medskip 

\noindent
Let us now assume that $n\geq 5$. As no element of $\E(E_\H)$ has odd cardinality, 
by Proposition~\ref{prop: parity joins as minimum cardinality fixed parity T-joins}, parity 
joins coincide with minimum cardinality $T$-joins. We will use this to show 
that $\mup(\H) = \lfloor\frac{n+1}{3}\rfloor +1$. Note that, by Proposition~\ref{prop: T sets of complete hypergraphs},
$\T(V_\H)=\P(V_\H)$. Given $T\subset V_\H$ let us denote by $\tau(\H,T)$ the minimum cardinality 
of a $T$-join. If $T=\emptyset$ then $\tau(\H,T) = 0$. If $|T|=1$, say $T=\set{v_1}$, then 
$\tau(\H,T)\leq 3$ since there exists a $T$-join of cardinality $3$, as is shown in Figure~\ref{fig: T joins in a complete graph}.
\begin{figure}[h]

\begin{tabular}{|c|c|}
\hline 

\begin{tikzpicture} [scale=1.4]
\clip (-1.2,-2.25) rectangle (1.4,.35);

\def\A{(-.75,-1.1)} 
\def\B{(.2,-1.5)} 
\def\C{(.9,-1)} 
\def\D{(.2,-.1)} 

\fill[green] \A -- \C -- \D -- \A;
\fill[blue] \A -- \B -- \C -- \A;
\draw[line width = .7pt] \A -- \C node[anchor = west]{$\ss v_3$};
\fill[yellow] \A -- \B -- \D -- \A;
\draw[line width = .7pt] \A -- \D node[anchor = south]{$\ss v_4$};
\draw[line width = .7pt] \A node[anchor = east]{$\ss v_1$} -- \B node[anchor = north]{$\ss v_2$};
\draw[line width = .7pt] \A -- \B;
\draw[line width = .7pt] \B -- \C;
\draw[line width = .7pt] \C -- \D;
\draw[line width = .7pt] \D -- \B;
\draw[line width = .5pt, dashed] \A -- \C;

\fill\A circle (1.5pt);
\fill[gray]\B circle (1.5pt);
\fill[gray]\C circle (1.5pt);
\fill[gray]\D circle (1.5pt);

\draw(1.1,-2.15) node {$\ss |T|=1$};
\end{tikzpicture}

&

\begin{tikzpicture} [scale=1.4]
\clip (-1.2,-2.25) rectangle (1.4,.35);

\def\A{(-.75,-1.1)} 
\def\B{(.2,-1.5)} 
\def\C{(.9,-1)} 
\def\D{(.2,-.1)} 

\fill[blue] \A -- \B -- \C -- \A;
\draw[line width = .7pt] \A -- \C node[anchor = west]{$\ss v_2$};
\fill[yellow] \A -- \B -- \D -- \A;
\draw[line width = .7pt] \A -- \D node[anchor = south]{$\ss v_1$};
\draw[line width = .7pt] \A node[anchor = east]{$\ss v_4$} -- \B node[anchor = north]{$\ss v_3$};
\draw[line width = .7pt] \A -- \B;
\draw[line width = .7pt] \B -- \C;
\draw[line width = .7pt] \D -- \B;
\draw[line width = .5pt, dashed] \A -- \C;

\fill[gray]\A circle (1.5pt);
\fill[gray]\B circle (1.5pt);
\fill[black]\C circle (1.5pt);
\fill[black]\D circle (1.5pt);

\draw(1.1,-2.15) node {$\ss |T|=2$};
\end{tikzpicture}

\\
\hline

\begin{tikzpicture} [scale=1.4]
\clip (-1.2,-2.25) rectangle (1.4,.35);

\def\A{(-.5,-1.25)} 
\def\B{(.5,-1.25)} 
\def\D{(0,-.25)} 

\fill[green] \A -- \B -- \D -- \A;
\draw[line width = .7pt] \A -- \D node[anchor = south]{$\ss v_1$};
\draw[line width = .7pt] \A node[anchor = east]{$\ss v_3$} -- \B node[anchor = west]{$\ss v_2$};
\draw[line width = .7pt] \A -- \B;
\draw[line width = .7pt] \D -- \B;

\fill\A circle (1.5pt);
\fill\B circle (1.5pt);
\fill\D circle (1.5pt);

\draw(1.1,-2.15) node {$\ss |T|=3$};
\end{tikzpicture}

&

\begin{tikzpicture} [scale=1.4]
\clip (-1.3,-2.35) rectangle (1.3,.25);

\def\A{(-.75,-.25)} 
\def\B{(-.75,-1.5)} 
\def\C{(.75,-.25)} 
\def\D{(.75,-1.5)} 
\def\E{(0,-.9)} 

\fill[blue] \A -- \E -- \B -- \A;
\fill[green] \C -- \D -- \E -- \E;
\draw[line width = .7pt] \E  -- \C node[anchor = west]{$\ss v_2$};
\draw[line width = .7pt] \E -- \D node[anchor = west]{$\ss v_3$};
\draw[line width = .7pt] \A node[anchor = east]{$\ss v_1$} -- \B node[anchor = east]{$\ss v_4$};
\draw[line width = .7pt] \A -- \E;
\draw[line width = .7pt] \B -- \C;
\draw[line width = .7pt] \C -- \D;
\draw \E+(0,.1) node[anchor = south]{$\ss v_5$};

\fill\A circle (1.5pt);
\fill\B circle (1.5pt);
\fill\C circle (1.5pt);
\fill\D circle (1.5pt);
\fill[gray]\E circle (1.5pt);

\draw(1.0,-2.25) node {$\ss |T|=4$};
\end{tikzpicture}

\\
\hline 
\end{tabular}

\caption{$T$-joins in $K^3_n$, for $0\leq |T|\leq 7$.}
\label{fig: T joins in a complete graph}
\end{figure}
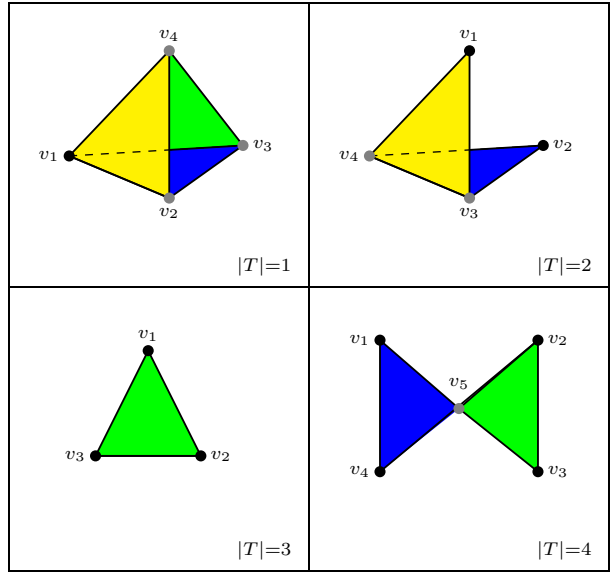
(In this figure, vertices in $T$ are shown in black, other vertices in gray.)
Assume now $|T|\geq 2$ and let us use induction to prove that 
$$\ts \tau(\H,T) \leq \lfloor \frac{|T|+1}{3}\rfloor +1.$$
For $|T|=2,3$ and $4$ this is proved by showing there is a $T$-join of cardinality $\leq 2$. For each $|T|$ 
this $T$-join is indicated in Figure~\ref{fig: T joins in a complete graph}.
Assume now that $|T|\geq 5$. Consider any three elements of $T$ say $v_1,v_2,v_3$, let 
$T'=T\setminus \set{v_1,v_2,v_3}$ and, by induction, let $J$ be a $T'$-join of cardinality 
less than or equal to $\lfloor \frac{|T|+1}{3}\rfloor$. Then 
$$
J = J\sd \set{\set{v_1,v_2,v_3}}
$$
is a $T$-join with cardinality less than or equal to $\lfloor \frac{|T|+1}{3}\rfloor +1$.
\medskip 

\noindent 
Since $|T|\leq n$ and $n\geq 5$, so far we have shown that a minimum cardinality $T$-join has cardinality less than 
or equal to $\lfloor \frac{n+1}{3}\rfloor+1$. Hence 
$\ts \mup(\H) \leq \lfloor \frac{n+1}{3}\rfloor+1$.
To show the opposite inequality we only need to prove that there exists $T\subset V_\H$ 
such that any $T$-join has cardinality at least $\lfloor \frac{n+1}{3}\rfloor+1$. 
Let $T\subset V_\H$ be any set of $3\lfloor\frac{n+1}{3}\rfloor-1$ vertices of $\H$ 
and let $J$ be a $T$-join. Then, from 
\begin{equation}\label{eq: L1491}
\ts \sum_{v\in T}\deg_{J}(v) + \sum_{v\not \in T}\deg_{J}(v) = 3|J|
\end{equation}
we get 
$$
\ts 3|J|\geq |T| = 3\lfloor\frac{n+1}{3}\rfloor-1 \implies |J|\geq \lfloor\frac{n+1}{3}\rfloor.
$$
Now, $|J|=\lfloor\frac{n+1}{3}\rfloor$ is impossible since \eqref{eq: L1491} reduced modulo $2$ yields 
$$
\ts |J|\equiv_2 |T| \equiv_2\lfloor\frac{n+1}{3}\rfloor -1.
$$
We conclude that $|J|\geq \lfloor\frac{n+1}{3}\rfloor +1$.
\end{proof}

\end{document}